\documentclass{amsart}
\usepackage{amssymb}
\usepackage{float}
\usepackage{comment}
\usepackage{array}
\usepackage[margin=1in]{geometry}
\makeatletter
\newcommand{\thickhline}{%
	\noalign {\ifnum 0=`}\fi \hrule height 1pt
	\futurelet \reserved@a \@xhline
}
\newcolumntype{"}{@{\hskip\tabcolsep\vrule width 1pt\hskip\tabcolsep}}
\makeatother

\newtheorem{theorem}{Theorem}[section]
\newtheorem{lemma}[theorem]{Lemma}

\theoremstyle{definition}

\newtheorem{corollary}[theorem]{Corollary}
\newtheorem{proposition}[theorem]{Proposition}
\newtheorem{example}[theorem]{Example}

\theoremstyle{remark}
\newtheorem{remark}[theorem]{Remark}

\numberwithin{equation}{section}

\begin{document}
	\title{On the Monotonocity of the Hilbert Functions for 4-generated pseudo-symmetric monomial curves}
	\author{N\.{i}l \c{S}ah\.{i}n}
	\address{Department of Industrial Engineering, Bilkent University, Ankara, 06800 Turkey}
	\email{nilsahin@bilkent.edu.tr}

	\thanks{}
	
	
	\date{\today}
	
	\commby{}
	
	\dedicatory{}

\begin{abstract} 
In this article  we solve the conjecture "Hilbert function of the local ring for a 4 generated pseudo-symmetric numerical semigroup $<n_1,n_2,n_3,n_4>$ is always non-decreasing when $n_1<n_2<n_3<n_4$". We give a complete characterization to the standard bases  when the tangent cone is not Cohen-Macaulay by showing that the number of elements in the standard basis depends on some parameters $s_j$'s we define. Since the tangent cone is not Cohen-Macaulay, non-decreasingness of the Hilbert fuction was not guaranteed, we proved the non-decreasingness from our explicit Hilbert Function computation.
\end{abstract}

\maketitle
\section{Introduction}
\label{Sec:1}
Cohen-Macaulayness and the Hilbert Functions of the tangent cone of a projective variety is a classical problem of commutative algebra as the Hilbert Function gives important geometric information like the degree, arithmetic genus and the dimension of the variety. Despite the fact that Hilbert function of a Cohen-Macaulay algebra is well understood, very little is known in the local case. One of the main problems is if the properties of the local ring can be carried out to the tangent cone, see \cite{Rossi}. It is well known that the good properties of the local ring as being Cohen-Macaulay, Gorenstein, complete intersection or level can not be carried out to the the tangent cone in general. Most of the time, to explore these, an explicit standard basis computation to the defining ideal of the local ring and of the graded ring is required, \cite{Leila}. A long standing conjecture in the theory of Hilbert functions is  Sally's conjecture " Hibert Function of a one dimensional Cohen-Macaulay local ring with small enough embedding dimension is nondecreasing" \cite{Sally}. The statement is obvious for embedding dimension 1, proved by Matlis \cite{Matlis} and Elias \cite{Elias1} in embedding dimensions 2 and 3 respectively. Counter examples are given for embedding dimension 4 by Gupta and Roberts, for each embedding dimension greter than 4 by Orecchia \cite{GuptaRoberts,Orecchia}. The problem is open even in monomial curve case: There are many affirmative answers but the counter examples were given in affine 10 space by Herzog and Waldi \cite{HerzogWaldi} and in affine 12 space by Eakin and Sathaye \cite{EakinSathaye}. In these counter examples, the local ring was Cohen-Macaulay so the Cohen-Macaulayess of the local ring does not guarantee the non-decreasingness property of the Hilbert Function when embedding dimension is greater than 3. In \cite{Rossi} Rossi conjectured that "Hilbert Function of a one dimensional Gorenstein local ring is non-decreasing". This problem has many positive answers, see \cite{am,AMS,PatilTamone,AOS,JafZar,AKN,Oneto}. Recently, in \cite{OnetoStraz} Oneto, Strazzanti and Tamone constructed explicit examples of Gorenstein numerical semigroup rings with decreasing Hilbert Function and give counter examples to Rossi's conjecture. However, Sally's conjecture is still open for the monomial curves in $n$ space when $3<n<10$. We focus on the first case: numerical semigroups in 4-space. In 1978, Stanley proved that "If the tangent cone is Cohen-Macaulay, then The Hilbert function of a Cohen-Macaulay local ring is nondecreasing". Arslan and Mete put the extra condition $\alpha_{2}\leq \alpha_{21}+\alpha_{24}$ to the generators of a symmetric semigroup when $n_1<n_2<n_3<n_4$. They find the generators of the defining ideal and showed the Cohen-Macaulayness of the tangentcone which proves the nondecreasingness of the Hilbert Function by Stanley's theorem. The case $\alpha_{2}> \alpha_{21}+\alpha_{24}$ is still open in symmetric case. As symmetric and pseudosymmetric semigroups are maximal with respect to inclusion with fixed genus, we focus on 4 generated pseudosymmetric monomial curves in this paper and solve the conjecture 
\begin{center}" Is the Hilbert function of the local ring corresponding to a 4 generated numerical semigroup nondecreasing?".\end{center}

\noindent The case $\alpha_2\leq \alpha_{21}+1$ is studied in \cite{SahinSahin} and since the tangent cone is C-M, without an explicit Hilbert Function computation, non-decreasingness of the Hilbert Function is proved by the Stanley's theorem. Though in this case there are 5 elements in the standard basis, the number of elements in the standard basis increase in the open case  $\alpha_2> \alpha_{21}+1$ when $\alpha_4$ increase, which makes the standard basis computation difficult as the normal forms of the s-polynomials should be added to the standard basis each time, see \cite{Sahin}. Furthermore, as the tangent cone is not C-M , an explicit Hilbert Function computation is needed in this case to prove the non-decreasingness of the Hilbert Function.  $\alpha_4=2$ and $\alpha_4=3$ cases are investigated in \cite{Sahin, Sahin2}  respectively. Though the Hilbert Functions are computed in both of these cases, nondecreasingness of the Hilbert Function is not proved. Since  the number of elements inside the standard basis increase when $\alpha_4$ increase, giving a complete characterization to the standard basis for a general $\alpha_4$ is significant in showing the nondecreasingness of the Hilbert function. In this paper,  we will give a standard basis for a general $\alpha_4$, show that the elements in this basis depend on some parameters we define as $s_j's$ for $j=0,1,\hdots, \alpha_4-1$. We show that the Hilbert Function is non-decreasing independent from the $s_j$'s.

The structure of the paper is the following. In Section \ref{prelim} we introduce pseudo-symmetric semigroups and give some preliminaries. In Section \ref{2}, we decribe the standard basis of the defining ideal in Theorem \ref{standardbasis}. In Section \ref{3} we decribe the Hilbert Function in Theorem \ref{Hilbfunc}, second Hilbert Function in Theorem \ref{2ndHilb} and prove our main result, see Theorem \ref{nondecreasing}. In section \ref{4}, we give explicit examples of 4-generated pseudo-symmetric monomial curves with nondecreasing Hilbert functions. Finally, the appendix contains some technical facts required to prove nondecreasingness of the Hilbert Function.
\section{Preliminaries}\label{prelim}
\noindent $n_1< n_2<\dots<n_k$ being positive integers with $\gcd (n_1,\dots,n_k)=1$, the numerical semigroup generated by these integers is defined as $S=\langle n_1,\dots,n_k \rangle=\{ \displaystyle\sum_{i=1}^{k} u_in_i | u_i \in \mathbb{N}\}$. $K$ being an algebraically closed field, the semigroup ring of $S$ is $K[S]=K[t^{n_1}, t^{n_2}, \dots, t^{n_k}]$ and let $A=K[X_1,X_2,\dots,X_k]$. If  $\phi: A {\longrightarrow} K[S]$ with $\phi(X_i)=t^{n_i}$ and $\ker \phi=I_S$ , then $K[S]\simeq A/I_S$. Let $C_S$ be the affine curve corresponding to $S$ with parametrization 
$$X_1=t^{n_1},\ \ X_2=t^{n_2},\ \dots,\  X_k=t^{n_k}  $$
then $I_S$ is called the defining ideal of $C_S$. The  \textit{multiplicity} of $C_S$ is smallest integer $n_1$ in the semigroup. Let's denote the corresponding local ring with $R_S=K[[t^{n_1},\dots,t^{n_k}]]$ and the maximal ideal with $\mathfrak{m}=\langle t^{n_1},\hdots,t^{n_k}\rangle$. Then
$gr_{\mathfrak{m}}(R_S)=\bigoplus_{i=0}^{\infty} \mathfrak{m}^i/\mathfrak{m}^{i+1}\cong A/{I^*_S},$ is the associated graded ring 
where ${I^*_S}=\langle f^*| f \in I_S \rangle$ with $f^*$ denoting the least homogeneous summand of $f$. 

The Hilbert function of the associated
graded ring $gr_{\mathfrak{m}}(R_S)=\bigoplus_{i=0}^{\infty} \mathfrak{m}^i/\mathfrak{m}^{i+1}$ is often referred to the Hilbert function $H_{R_S}(n)$ of the local ring $R_S$.
In other words,
$$H_{R_S}(n)=H_{gr_{\mathfrak{m}}(R_S)}(n)=dim_{R_S/\mathfrak{m}}(\mathfrak{m}^n/\mathfrak{m}^{n+1}) \; \; n\geq
0.$$
This function is called non-decreasing if $H_{R_S}(n)\geq H_{R_S}(n-1)$ for all $n \in \mathbb{N}$.
The Hilbert series of $R_S$ is defined to be the generating function
$$HS_{R_S}(t)=\begin{displaystyle}\sum_{n \in \mathbb{N}}\end{displaystyle}H_{R_S}(n)t^n.$$
By the Hilbert-Serre theorem it can also be written as:
$HS_{R_S}(t)=\frac{P(t)}{(1-t)^k}=\frac{Q(t)}{(1-t)^d}$, where $P(t)$ and $Q(t)$ are polynomials with coefficients in
$\mathbb{Z}$ and $d$ is the Krull dimension of $R_S$. $P(t)$ is called first Hilbert Series and $Q(t)$ is called second Hilbert series, \cite{greuel-pfister,Rossi}. It is also known that there is
a polynomial $P_{R_S}(n) \in \mathbb{Q}[n]$ called Hilbert polynomial of $R_S$ such that
$H_{R_S}(n)=P_{R_S}(n)$ for all $n \geq n_0$, for some $n_0 \in \mathbb{N}$. The smallest $n_0$ satisfying this
condition is the regularity index of the Hilbert function of $R_S$.

In \cite{komeda}, Komeda gives an explicit description to 4- generated pseudo-symmetric numerical semigroups : A $4$-generated  semigroup $S=\langle n_1,n_2,n_3,n_4 \rangle$ is pseudo-symmetric if and only if there are integers $\alpha_i>1$, for
$1\le i\le4$, and $\alpha_{21}>0$ with $0<\alpha_{21}<\alpha_1-1$,
such that 
\begin{eqnarray*}
	n_1&=&\alpha_2\alpha_3(\alpha_4-1)+1,\\
	n_2&=&\alpha_{21}\alpha_3\alpha_4+(\alpha_1-\alpha_{21}-1)(\alpha_3-1)+\alpha_3,\\
	n_3&=&\alpha_1\alpha_4+(\alpha_1-\alpha_{21}-1)(\alpha_2-1)(\alpha_4-1)-\alpha_4+1,\\
	n_4&=&\alpha_1\alpha_2(\alpha_3-1)+\alpha_{21}(\alpha_2-1)+\alpha_2. 
\end{eqnarray*} He also gave an explicit characterization to the toric ideal as  $I_S=\langle f_1,f_2,f_3,f_4,f_5 \rangle$ with
\begin{eqnarray*} f_1&=&X_1^{\alpha_1}-X_3X_4^{\alpha_4-1},  \quad \quad
	f_2=X_2^{\alpha_2}-X_1^{\alpha_{21}}X_4, \quad
	f_3=X_3^{\alpha_3}-X_1^{\alpha_1-\alpha_{21}-1}X_2,\\
	f_4&=&X_4^{\alpha_4}-X_1X_2^{\alpha_2-1}X_3^{\alpha_3-1}, \quad 
	f_5=X_1^{\alpha_{21}+1}X_3^{\alpha_3-1}-X_2X_4^{\alpha_4-1}.
\end{eqnarray*}  
If $n_1<n_2<n_3< n_4$ then it is known from \cite{SahinSahin} that 
\begin{enumerate}
	\item[(1)] $\alpha_1>\alpha_4$
	\item[(2)] $\alpha_3<\alpha_1-\alpha_{21}$
	\item[(3)] $\alpha_4<\alpha_2+\alpha_3-1$
\end{enumerate}
and these conditions completely determine the leading monomials of $f_1, f_3$ and $f_4$. Indeed, $ {\rm LM}( f_1)= X_3X_4^{\alpha_{4}-1}$ by $(1)$, ${\rm LM}(f_3)= X_3^{\alpha_3}$ by $(2)$, ${\rm LM}(f_4)=X_4^{\alpha_4}$ by $(3)$. For the case $\alpha_2\leq\alpha_{21}+1$, a we have given a complete characterization to the standard basis in \cite{SahinSahin} and since the tangent cone is Cohen-Macaulay in this case, we showed that the Hilbert funciton is nondecreasing. 
If we let
\begin{enumerate}
	\item[(4)]  $\alpha_2>\alpha_{21}+1$
\end{enumerate}
we determine the leading monomial of $f_2$ as ${\rm LM}( f_2)=X_1^{\alpha_{21}}X_4$.
\section{Standard bases}\label{2}
Before we state and prove our main theorem, we will prove the following proposition about the normal forms of the polynomials that are in specific forms to simplify our computations. This may be a proposition that is stated and proved before but since we did not encounter it, we will give a proof here.
\begin{proposition}\label{shortcut}
	Let $m_1$ and $m_2$ be monomials and let $g=m_1-m_2$, $f=m_1^k-m_2^k$ where $k $ is a natural number greater than or equal to $1$. If $G$ is a set containing $g$, then ${\rm NF}(f|G)=0$.
\end{proposition}

\begin{proof}
	Without losing generality, assume that ${\rm LM }(g)=m_1$ which makes ${\rm LM}(f)=m_1^k$. Then ${\rm spoly}(f,g)=f-m_1^{k-1}g=m_2\left[  m_1^{k-1}-m_2^{k-1}\right] =r_1$. ${\rm LM}(r_1)=m_1^{k-1}m_2$ and ${\rm spoly}(r_1,g)=m_1^{k-2}m_2g-r_1=-m_2^2\left[ m_1^{k-2}-m_2^{k-2} \right] =r_2$. Then continuing inductively, $r_k={\rm spoly }(r_{k-1},g)=(-1)^{k+1}m_2^kg$ and hence, ${\rm spoly}(r_k,g)=r_k-(-1)^{k+1}m_2^kg=0$
\end{proof}
Recall the next remark from \cite{cox}
\begin{remark}\label{cox}
	Let $f$ be a polynomial and $G$ be a standard basis of an ideal $I$. Then $f\in I$ iff ${\rm NF}(f|g)=0$
\end{remark}
\begin{theorem}\label{standardbasis}
	Let $S=\langle n_1,n_2,n_3,n_4 \rangle$ be a 4-generated pseudosymmetric numerical semigroup with $n_1<n_2<n_3<n_4$. If $\alpha_2>\alpha_{21}+1$ then the standard basis for $I_S$ is
	$$G=\{f_{1,j},f_2,f_3,f_4,g_{j,i_j} \}$$
	where 
$f_{1,j}=X_1^{\alpha_1+j\alpha_{21}}-X_2^{j\alpha_2}X_3X_4^{\alpha_4-(j+1)}$, $j=0,1,\hdots, \alpha_4-1$

$g_{j,i_j}=X_2^{((\alpha_4-1)i_j+j)\alpha_2+1}X_4^{\alpha_4-(j+1)}-X_1^{i_j\alpha_1+((\alpha_4-1)i_j+j+1)\alpha_{21}+1}X_3^{\alpha_3-(i_j+1)}$ $j=0,1,2,\hdots,\alpha_4-1$  and $i_j=s_{j-1}, s_{j-1}+1, s_{j-1}+2,\hdots,s_j-1, s_j$. $s_{-1}=0$ and $s_j$ is the smallest integer with $((\alpha_4-1)s_j+j)\alpha_2+\alpha_4-j< s_j\alpha_1+((\alpha_4-1)s_j+j+1)\alpha_{21}+\alpha_3-s_j$
\end{theorem}
To prove the theorem, we will need the following lemmas and remark.

	\begin{lemma}\label{lmf1j} 
		$j\alpha_2+\alpha_4< \alpha_1+j\alpha_{21}+j$ for all $j=0,1,\hdots,\alpha_4-1$ and hence ${\rm LM}(f_{1,j})=X_2^{j\alpha_2}X_3X_4^{\alpha_4-(j+1)}$ for all $j$, too.
\end{lemma}
\begin{proof}
	Since $n_1<n_2$, we have 
	\begin{eqnarray*}
		\alpha_2\alpha_3(\alpha_4-1)+1&<&\alpha_{21}\alpha_3\alpha_4+(\alpha_1-\alpha_{21}-1)(\alpha_3-1)+\alpha_3 \\
		\alpha_1-\alpha_{21}+\alpha_3(\alpha_{21}+\alpha_2\alpha_4)&<&\alpha_3(\alpha_2+\alpha_{21}\alpha_4+\alpha_1) \\
		\alpha_1-\alpha_3-\alpha_{21}+\alpha_3(\alpha_{21}+\alpha_2\alpha_4+1)&<&\alpha_3(\alpha_2+\alpha_{21}\alpha_4+\alpha_1)
	\end{eqnarray*}
Since $\alpha_1-\alpha_3-\alpha_{21}>0$, we have $\alpha_3(\alpha_{21}+\alpha_2\alpha_4+1)<\alpha_3(\alpha_2+\alpha_{21}\alpha_4+\alpha_1)$, and cancelling out $\alpha_3$, we obtain $$\alpha_2(\alpha_4-1)+1< \alpha_{21}(\alpha_4-1)+\alpha_1$$
 or equivalently, $(\alpha_2-\alpha_{21})(\alpha_4-1)+1 < \alpha_1$. To obtain the same inequality for $j<\alpha_4-1$, 
 \begin{eqnarray*}
 	 (\alpha_2-\alpha_{21})(\alpha_4-1)+1 &<& \alpha_1\\
 	  (\alpha_2-\alpha_{21}-1)(\alpha_4-1)+\alpha_4 &<& \alpha_1\\
 \end{eqnarray*}
Since $j<\alpha_4-1$ and $\alpha_2-\alpha_{21}-1>0$, we have $(\alpha_2-\alpha_{21}-1)j+\alpha_4 < \alpha_1$ or equivalently,
$$j\alpha_2+\alpha_4<\alpha_1+j\alpha_{21}+j$$
\end{proof}

Note that lemma \ref{lmf1j} is the generalization of remark (1.1) of \cite{Sahin}  and remark (3.1) of \cite{Sahin2}.
\begin{lemma}\label{NFg}
${\rm NF} (g_{j,m}|G)=0$ for any $m=0,1,\hdots, s_{\alpha_4-1}$	
\end{lemma}
\begin{proof}
	
	For $s_{j-1}\leq m\leq s_j$, $g_{j,m}\in G$ and hence the result is clear. 
	
	For $m > s_{j}$,	$T_{g_{j,m}}=\{g_{j,s_{j}}\}$ and ${\rm spoly}(g_{j,m},g_{j,s_{j}})=$ $X_1^{(s_{j})\alpha_1+((\alpha_4-1)(s_{j})+j+1)\alpha_{21}+1}X_3^{\alpha_3-(m+1)}$ $\left[ (X_1^{\alpha_1+(\alpha_4-1)\alpha_{21}})^{m-s_{j}}\right. $ $\left. -(X_2^{(\alpha_4-1)\alpha_{21}}X_3)^{m-s_{j}} \right]$ $=r_1 $. Since the monomials inside the paranthesis are $m-s_j$th powers of the monomials in $f_1,\alpha_4-1$, by lemma \ref{shortcut} and remark \ref{cox}, ${\rm NF} (g_{j,m}|G)=0$.

	For $m<s_{j-1}$, ${\rm NF} (g_{j,m}|G)=0$ as ${\rm spoly}(g_{j,m},g_{j-1,m})=X_2^{((\alpha_4-1)m+j-1)\alpha_2+1}X_4^{\alpha_4-(j+1)}f_2$ and remark \ref{cox} gives the result.
\end{proof}
\begin{remark}\label{j=0}
	If $j=0$ then $0$ is the smallest integer satisfying $((\alpha_4-1)s_j+j)\alpha_2+\alpha_4-j< s_j\alpha_1+((\alpha_4-1)s_j+j+1)\alpha_{21}+\alpha_3-s_j$ by remark 1.1 of \cite{Sahin}. That is $s_0=0$. Futhermore, when $j=0$ then $g_{0,s_0}=g_{0,0}=f_5$ in \cite{komeda}.
\end{remark}

\begin{proof}
We will use NFM\tiny{ORA} \normalsize as the normal form.
	\begin{itemize}
		\item $NF({\rm spoly}(f_2,f_3) \vert G)=0$, $NF({\rm spoly}(f_3,f_4) \vert G)=0$, $NF({\rm spoly}(f_3,f_5) \vert G)=0$ as the leading monomials are relatively prime.	
		
		\item $NF({\rm spoly}(f_2,f_4) \vert G)=0$ as ${\rm spoly}(f_2,f_4)=X_2^{\alpha_2-1}g_{0,0}$.
		
		
	\end{itemize}
Normal forms of the s-polynomials with $f_{1,j}$ will be investigated in two cases: When $j=\alpha_4-1$ and $j<\alpha_4-1$ since the leading monomial of $f_{1,j}$ change in these two cases. 
\begin{itemize}
	\item If $j=\alpha_4-1$, then $NF({\rm spoly}(f_{1,{\alpha_4-1}},f_2) \vert G)=0$ and $NF({\rm spoly}(f_{1,\alpha_4-1},f_4) \vert G)=0$ as the leading monomials are relatively prime. , 
	
	\noindent${\rm NF}({\rm spoly}(f_{1,\alpha_4-1},f_3) \vert G)=0$ as ${\rm spoly}(f_{1,\alpha_4-1},f_3)=X_1^{\alpha_1-\alpha_{21}-1}g_{\alpha_4-1,0}$ and lemma \ref{NFg} gives the result. 
	
	\item If $j<\alpha_4-1$, then
	
	 ${\rm NF}({\rm spoly}(f_{1,j},f_2) \vert G)=0$ as ${\rm spoly}(f_{1,j},f_2)=f_{1,j+1}\in G$.
	
	 ${\rm NF}({\rm spoly}(f_{1,j},f_3) \vert G)=0$ as  ${\rm spoly}(f_{1,j},f_3)=X_1^{\alpha_1-\alpha_{21}-1}g_{j,0}$ and $g_{j,0}\in G$
	 
	  ${\rm NF}({\rm spoly}(f_{1,j},f_4) \vert G)=0$, as ${\rm spoly}(f_{1,j},f_4)=X_1(X_2^{(j+1)\alpha_2-1}X_3^{\alpha_3}-X_1^{\alpha_1+j\alpha_{21}-1}X_4^{j+1})=r_1$. 
	  \begin{itemize}
	  	\item If ${\rm LM}(r_1)=X_1X_2^{(j+1)\alpha_2-1}X_3^{\alpha_3}$ then $T_{r_1}=\{f_3\}$ and ${\rm spoly}(f_3,r_1)=X_1^{\alpha_1-\alpha_{21}}\left[ X_2^{(j+1)\alpha_2}-X_1^{(j+1)\alpha_{21}}X_4^{j+1}\right]=r_2 $. Since the monomials inside the paranthesis are $j+1$th powers of the monomials of $f_2$, proposition \ref{shortcut} and remark \ref{cox} gives the result.     
  	    \item If ${\rm LM}(r_1)=X_1^{\alpha_1+j\alpha_{21}}X_4^{j+1}$ then $T_{r_1}=\{f_2\}$ and ${\rm spoly}(f_2,r_1)=X_1X_2^{\alpha_2}\left[X_1^{\alpha_1+(j-1)\alpha_{21}-1}X_4^{j}-X_2^{j\alpha_2-1}X_3^{\alpha_3} \right]=r_2 $. Continuing inductively, $T_{r_{j+1}}=\{f_2\}$ and $r_{j+2}={\rm spoly}(f_3,r_{j+1})=X_1X_2^{(j+2)\alpha_2-1}f_3$. Then remark \ref{cox} gives the result.
	  \end{itemize}
	  

	\item  ${\rm NF}({\rm spoly}(f_{1,j_1},f_{1,j_2}) \vert G)=0$ for all $0\leq j_1<j_2\leq \alpha_4-1$. Indeed,  proposition \ref{shortcut} and remark \ref{cox} and the fact that ${\rm spoly}(f_{1,j_1},f_{1,j_2})=X_1^{\alpha_1+j_1\alpha_{21}}\left[X_2^{(j_2-j_1)\alpha_2}-X_1^{(j_2-j_1)\alpha_{21}}X_4^{j_2-j_1} \right] $, gives the result.
	
\end{itemize}

\noindent Normal forms of the s-polynomials with $g_{j,{i_j}}$ will be investigated in two cases: When $i_j=s_j$ and $i_j<s_j$ since the leading monomial of $g_{j,{i_j}}$ change in these two cases.
\begin{itemize}
	\item If  $i_j<s_j$:
	
 ${\rm NF}({\rm spoly}(g_{j,i_j},f_4) \vert G)=0$ ,  $NF({\rm spoly}(g_{j,i_j},g_{j,s_j}) \vert G)=0$ as the leading monomials are relatively prime.

${\rm NF}({\rm spoly}(g_{j,i_j},f_2)|G)=0$ as $ {\rm spoly}(g_{j,i_j},f_2)=X_2^{\alpha_2}g_{j-1,i_j}$ and since ${\rm NF} (g_{j-1,i_j}|G)=0$ by lemma \ref{NFg}.

${\rm NF}({\rm spoly}(g_{j,i_j},f_3)|G)=0$ as ${\rm spoly}(g_{j,i_j},f_3)=X_2\left[ X_1^{(i_j+1)\alpha_1+((\alpha_4-1)i_j+j)\alpha_{21}}-X_2^{((\alpha_4-1)i_j+j)\alpha_{2}}X_3^{i_j+1}X_4^{\alpha_4-(j+1)}\right]$ $=r_1 $.
  ${\rm LM}(r_1)=X_2^{((\alpha_4-1)i_j+j)\alpha_{2}+1}X_3^{i_j+1}X_4^{\alpha_4-(j+1)}$ by lemma \ref{lmf1j} and $T_{r_1}=\{f_{1,j}\}$. Then 
	
	${\rm spoly}(f_{1,j},r_1)=X_1^{\alpha_1+j\alpha_{21}}X_2\left[ X_1^{i_j\alpha_1+(\alpha_4-1)i_j\alpha_{21}}-X_2^{(\alpha_4-1)i_j\alpha_2}X_3^{i_j} \right]=r_2 $. Since the monomials inside the paranthesis are $i_j$th powers of the monomials of $f_{1,\alpha_4-1}$, the result follows from proposition \ref{shortcut} and remark \ref{cox}

	\item If  $i_j=s_j$:
	
${\rm NF}({\rm spoly}(g_{j,s_j},f_2)|G)=0$ as ${\rm spoly}(g_{j,s_j},f_2)=g_{j+1,s_j} \in G$
	
${\rm NF}({\rm spoly}(g_{j,s_j},f_3)|G)=0$ as the leading monomials are relatively prime.

${\rm NF}({\rm spoly}(g_{j,s_j},f_4)|G)=0$ as ${\rm spoly}(g_{j,s_j},f_4)=X_1^{}X_3^{\alpha_3-(s_j+1)}\left[ X_1^{s_j\alpha_1+((\alpha_4-1)s_j+j+1)\alpha_{21}}X_4^{j+1}-\right. $ $\left. X_2^{((\alpha_4-1)s_j+j+1)\alpha_2}X_3^{s_j}\right]=r_1 $. 
\begin{itemize}
	\item If  ${\rm LM}(r_1)=X_1^{s_j\alpha_1+((\alpha_4-1)s_j+j+1)\alpha_{21}+1}X_3^{\alpha_3-(s_j+1)}X_4^{j+1}$, then  $T_{r_1}=\{f_{2}\}$ and ${\rm spoly}(r_1,f_{2})=X_1^{}X_2^{\alpha_2}X_3^{\alpha_3-(s_j+1)}\left[ X_1^{s_j\alpha_1+((\alpha_4-1)s_j+j)\alpha_{21}}X_4^{j}-X_2^{((\alpha_4-1)s_j+j)\alpha_2}X_3^{s_j}\right] =r_2$. Continuing inductively, $r_{j+2}={\rm spoly}(r_{j+1},f_{2} )=X_1X_2^{(j+1)\alpha_2}X_3^{\alpha_3-{s_j+1}}\left[ X_1^{s_j\alpha_1+((\alpha_4-1)s_j)\alpha_{21}}-X_2^{((\alpha_4-1)s_j)\alpha_2}X_3^{s_j}\right] $. Since the monomials inside the paranthesis are $s_j$th powers of the monomials of $f_{1,\alpha_4-1}$, proposition \ref{shortcut} with remark \ref{cox} gives the result.
	\item If ${\rm LM}(r_1)= X_1X_2^{((\alpha_4-1)s_j+j+1)\alpha_2}X_3^{\alpha_3-1}$, then  $T_{r_1}=\{f_{1,\alpha_4-1}\}$ and ${\rm spoly}(r_1,f_{1,\alpha_4-1})=$
	\newline $X_1^{\alpha_1+(\alpha_4-1)\alpha_{21}+1}X_3^{\alpha_3-(s_j+1)}\left[ X_2^{((\alpha_4-1)(s_j-1)+j+1)\alpha_2}X_3^{s_j-1}-X_1^{(s_j-1)\alpha_1+((\alpha_4-1)(s_j-1)+j+1)\alpha_{21}}X_4^{j+1}\right] =r_2$. Continuing inductively,
	
	 $r_{s_j+1}={\rm spoly}(r_{s_j},f_{1,\alpha_4-1} )= X_1^{(\alpha_1+(\alpha_4-1)\alpha_{21})s_j+1}X_3^{\alpha_3-(s_j+1)}  \left[ X_2^{(j+1)\alpha_2}-X_1^{(j+1)\alpha_{21}X_4^{j+1}}\right] $. Since the monomials inside the paranthesis are $j+1$th powers of the monomials of $f_{2}$, proposition \ref{shortcut} with remark \ref{cox} gives the result.
\end{itemize}

\item ${\rm NF}({\rm spoly}(g_{j,i_{j}},f_{1,{j}})|G)=0$ for $j=0,1,\hdots,\alpha_4-1$. Indeed,

${\rm spoly}(g_{j,i_{j}},f_{1,{j}})=X_2^{((\alpha_4-1)i_j+2j)\alpha_2+1}X_4^{2\alpha_4-2(j+1)}-X_1^{(i_j+1)\alpha_1+((\alpha_4-1)i_j+2j+1)\alpha_{21}+1}X_3^{\alpha_3-(i_j+2)} =r_1$.  Since $i_j\geq s_{j-1}$, by the definition of $s_j$ we have 
$$((\alpha_4-1)i_j+j-1)\alpha_2+\alpha_4-j+1< i_j\alpha_1+((\alpha_4-1)i_j+j)\alpha_{21}+\alpha_3-i_j$$
Also, Lemma \ref{lmf1j} gives, 
$$(j+1)\alpha_2+\alpha_4<\alpha_1+(j+1)\alpha_{21}+j+1$$
Adding these two we obtain
${\rm LM}(r_1)=X_2^{((\alpha_4-1)i_j+2j)\alpha_2+1}X_4^{2\alpha_4-2(j+1)}$ and $T_{r_1}=\left\lbrace g_{0,0}\right\rbrace $. Then $r_2={\rm spoly}(r_1,g_{0,0})= X_1^{(i_j+1)\alpha_1+((\alpha_4-1)i_j+2j+1)\alpha_{21}+1}X_3^{\alpha_3-(i_j+2)}-X_1^{\alpha_{21}+1}X_2^{((\alpha_4-1)i_j+2j)\alpha_2}X_3^{\alpha_3-1}X_4^{\alpha_4-2j-1}$. Then $T_{r_2}=\left\lbrace f_{1,2j}\right\rbrace $ and $r_3={\rm spoly}(r_2,f_{1,2j})=X_1^{\alpha_1+(2j+1)\alpha_{21}+1}X_3^{\alpha_3-(i_j+2)}\left[ X_1^{i_j\alpha_1+(\alpha_4-1)i_j\alpha_{21}}-X_2^{(\alpha_4-1)i_j\alpha_2}X_3^{i_j} \right] $. Since the monomials inside the paranthesis are $i_j$th powers of the monomials inside $f_{1,\alpha_4-1}$, the result follows from remark \ref{cox} and proosition \ref{shortcut}.

\item ${\rm NF}({\rm spoly}(g_{j,s_j},g_{j,i_j})|G)=0$ as the leading monomials are relatively prime.
\item ${\rm NF}({\rm spoly}(g_{j,s_j},f_{1,j})|G)=0$ as ${\rm spoly}(g_{j,s_j},f_{1,j})=X_1^{\alpha_1+j\alpha_{21}}g_{\alpha_4-1,s_{j-1}}$,  remark \ref{cox} and lemma \ref{NFg} gives the result. 
\item ${\rm NF}({\rm spoly}(g_{j,s_j},f_{1,\alpha_4-1})|G)=0$ as ${\rm spoly}(g_{j,s_j},f_{1,\alpha_4-1})=X_1^{\alpha_1+(\alpha_4-1)\alpha_{21}}g_{j,s_j-1}$ and remark \ref{cox} gives the result.

\item ${\rm NF}({\rm spoly}(g_{j,n},g_{j,m})|G)=0$ for any $n<m<s_j$. Indeed, ${\rm spoly}(g_{j,n},g_{j,m})=X_2^{((\alpha_4-1)n+j)\alpha_2+1}X_4^{\alpha_4-(j+1)}\left[ \right. $ $\left.  X_1^{(\alpha_1+(\alpha_4-1)\alpha_{21})(m-n)}-X_2^{(\alpha_4-1)\alpha_{21}(m-n)}X_3^{m-n}\right] $ and monomials inside the paranthesis are $(m-n)$th powers of the monomials of $f_{1,\alpha_4-1}$. Then the result follows from proposion \ref{shortcut} and remark \ref{cox}.

\item ${\rm NF}({\rm spoly}(g_{j_1,s_{j_1}},g_{j_2,i_{j_2}})|G)=0$ as the leading monomials are relatively prime.

\item ${\rm NF}({\rm spoly}(g_{j_1,i_{j_1}},g_{j_2,i_{j_2}})|G)=0$ since we have ${\rm spoly}(g_{j_1,i_{j_1}},g_{j_2,i_{j_2}})=X_2^{((\alpha_4-1)i_{j_1}+j_1)\alpha_2+1}X_4^{\alpha_4-(j_2+1)} \left[  \right. $ \newline $\left.  X_1^{(i_{j_2}-i_{j_1})\alpha_1+((\alpha_4-1)(i_{j_2}-i_{j_1})+(j_2-j_1))\alpha_{21}}X_4^{j_2-j_1}- X_2^{((\alpha_4-1)(i_{j_2}-i_{j_1})+(j_2-j_1))\alpha_2}X_3^{i_{j_2}-i_{j_1}}\right]=r_1 $ and ${\rm LM}(r_1)=X_1^{(i_{j_2}-i_{j_1})\alpha_1+((\alpha_4-1)(i_{j_2}-i_{j_1})+(j_2-j_1))\alpha_{21}}  X_2^{((\alpha_4-1)i_{j_1}+j_1)\alpha_2+1}  X_4^{\alpha_4-(j_1+1)}$ 
then $T_{r_1}=\{f_2\}$ and ${\rm spoly}(r_1,f_2)=X_2^{((\alpha_4-1)i_{j_1}+j_1+1)\alpha_2+1}X_4^{\alpha_4-(j_2+1)} \left[  \right. $ $\left.  X_1^{(i_{j_2}-i_{j_1})\alpha_1+((\alpha_4-1)(i_{j_2}-i_{j_1})+(j_2-(j_1+1)))\alpha_{21}}X_4^{j_2-(j_1+1)}-\right.$ \newline$ \left.  X_2^{((\alpha_4-1)(i_{j_2}-i_{j_1})+(j_2-(j_1+1)))\alpha_2}X_3^{i_{j_2}-i_{j_1}}\right]=r_2 $. Continuing inductively, $r_{j_2-j_1+1}={\rm spoly}(r_{j_2-j_1},f_2)=X_2^{((\alpha_4-1)i_{j_1}+j_2)\alpha_2+1}X_4^{\alpha_4-(j_2+1)} \left[ X_1^{(i_{j_2}-i_{j_1})\alpha_1+((\alpha_4-1)(i_{j_2}-i_{j_1}))\alpha_{21}}- X_2^{((\alpha_4-1)(i_{j_2}-i_{j_1}))\alpha_2}X_3^{i_{j_2}-i_{j_1}}\right]$. Since the monomials inside the paranthesis are $(i_{j_2}-i_{j_1})$th powers of the monomials of $f_{1,\alpha_4-1}$, the result follows by proposition \ref{shortcut} and remark \ref{cox}.

\item${\rm NF}({\rm spoly}(g_{j_1,s_{j_1}},f_{1,{j_2}})|G)=0$  as ${\rm spoly}(g_{j_1,s_{j_1}},f_{1,{j_2}})=X_1^{\alpha_1+j_2\alpha_{21}}\left[ X_1^{(s_{j_1}-1)\alpha_1+((\alpha_4-1)s_{j_1}+j_1-j_2+1)\alpha_{21}+1}X_3^{\alpha_3-s_{j_1}}-\right.$ $ \left. X_2^{((\alpha_4-1)s_{j_1}+j_1-j_2)\alpha_2+1}X_4^{j_2-j_1}\right] =r_1$. 
\begin{itemize}
	\item If ${\rm LM}(r_1)=X_1^{\alpha_1+j_2\alpha_{21}}X_2^{((\alpha_4-1)s_{j_1}+j_1-j_2)\alpha_2+1}X_4^{j_2-j_1}$, then $T_{r_1}=\left\lbrace f_2\right\rbrace $ and ${\rm spoly}(r_1,f_2)=X_1^{\alpha_1+(j_2-1)\alpha_{21}}\left[ X_1^{(s_{j_1}-1)\alpha_1+((\alpha_4-1)s_{j_1}+j_1-j_2+2)\alpha_{21}+1}X_3^{\alpha_3-s_{j_1}}-\right.$ $ \left. X_2^{((\alpha_4-1)s_{j_1}+j_1-j_2+1)\alpha_2+1}X_4^{j_2-j_1-1}\right] =r_2$. Continuing inductively, $r_{j_2-j_1-1}=X_1^{\alpha_1+j_1\alpha_{21}}g_{\alpha_4-1,s_j-1}$ then remark \ref{cox} gives the result.
	
	\item If ${\rm LM}(r_1)= X_1^{(s_{j_1})\alpha_1+((\alpha_4-1)s_{j_1}+j_1+1)\alpha_{21}+1}X_3^{\alpha_3-s_{j_1}}$, $T_{r_1}=\{ g_{j_1,s_{j_1}-1}\}$ and ${\rm spoly}(r_1, g_{j_1,s_{j_1}-1})=X_1^{\alpha_1+j_2\alpha_{21}}X_2^{(\alpha_4-1)(s_{j_1-1}+j_1-j_2)\alpha_2+1}X_4^{j_2-j_1}\left[X_2^{(\alpha_4-1-j_2)\alpha_2}-X_1^{(\alpha_4-1-j_2)\alpha_{21}}X_4^{(\alpha_4-1-j_2)} \right] $. Since the monomials appearing inside the paranthesis are $(\alpha_4-1-j_2)$ th powers of the monomials in $f_2$, the result follows from lemma \ref{shortcut} and remark \ref{cox}.

\end{itemize}

\end{itemize}
\end{proof}

\begin{corollary}
	$\{{f_{1,j}}_*,{f_2}_*,{f_3}_*,{f_4}_*,{g_{j,i_j}}_*\}$ is a standard basis for ${I_S}_*$ for $j=0,1,\hdots,\alpha_4-1$ and $i_j=s_{j-1},\hdots,s_j$ where ${f_{1,j}}_*=X_2^{j\alpha_2}X_3X_4^{\alpha_4-(j+1)}$, 
	${f_2}_*=X_1^{\alpha_{21}}X_4$, 
	${f_3}_*=X_3^{\alpha_3}$, 
	${f_4}_*=X_4^{\alpha_4}$ and 
	${g_{j,i_j}}_*=X_1^{i_j\alpha_1+((\alpha_4-1)i_j+j+1)\alpha_{21}+1}X_3^{\alpha_3-(i_j+1)}$ for $i_j<s_j$ and ${g_{j,s_j}}_*=X_2^{((\alpha_4-1)s_j+j)\alpha_2+1}X_4^{\alpha_4-(j+1)}$. Since $X_1 | {f_2}_*$ the tangent cone is not Cohen-Macaulay.
\end{corollary}

\section{Hilbert Function}\label{3}
In this section, we show that, altough the the tangent cone of the 4-generated pseudosymmetric numerical semigroup given in theorem \ref{standardbasis} is non-Cohen-Macaulay, it has a non-decreasing Hilbert function. 
\begin{theorem}\label{Hilbfunc}
Let $P(I_S*)$ denote the numerator of the Hilbert series of the local ring $R_S$. Then

$P(I_S)_*=1-t^{\alpha_4}-t^{\alpha_4}(1-t)(1-t^{\alpha_{21}})-t^{((\alpha_4-1)s_{\alpha_4-1}+\alpha_4-1)\alpha_2+1}(1-t)^2-(1-t)\displaystyle\sum_{j=1}^{\alpha_4-1}t^{j\alpha_2+(\alpha_4-j)}-t^{\alpha_{21}+1}\left[ 1-t^{(\alpha_4-2)\alpha_2+1}-(1-t^{\alpha_2})\sum_{j=0}^{\alpha_4-3} t^{j\alpha_2+(\alpha_4-1-j)}\right] -(1-t)^2(1-t^{\alpha_{21}}) \sum_{j=0}^{\alpha_4-2}t^{((\alpha_4-1)s_{\alpha_j}+j)\alpha_2+\alpha_4-j} -$\newline$ (1-t^{(\alpha_4-1)\alpha_2})(1-t)^2\left[ \displaystyle\sum_{i=1}^{\alpha_4-1}  \displaystyle\sum_{j=s_{(i-1)}}^{s_i-1}t^{j\alpha_1+\left( (\alpha_4-1)j+(i+1)\right)\alpha_{21}+\alpha_3-j } \right] -\displaystyle t^{\alpha_3}\left[ (1-t^{\alpha_{21}+1})(1-t^{(\alpha_4-1)\alpha_2})-\right. $\newline$\displaystyle\left.  (1-t^{\alpha_2})(1-t^{\alpha_{21}})\sum_{j=0}^{\alpha_4-2}t^{j\alpha_2+(\alpha_4-1-j)}\right]$
\end{theorem}
\begin{proof}
	We use the following algorithm of Bayer and Stillman
	"If $I$ is a monomial ideal with $I =< J,w >$, then the numerator of the Hilbert series of $A/I$ is $P(I) = P(J)- t ^{\deg w} P(J : w)$ where $w$ is a monomial and $\deg w$ is the total degree of $w$.”
	that appears in \cite{bayer}. Though the order in which we choose the monomials inside $I_S*$  as $w$ does not matter, we picked them  as follows:
	$w_i=g_{\alpha_4-i,s_{\alpha_4-i}}*$ where $i=1,\hdots,\alpha_4$; $w_i=g_{k,k_j}*$ where $k=\alpha_4-1,\alpha_4-2,\hdots,2,1$ and for each $k$, $k_j=s_{k}-1,s_{k}-2,\hdots,s_{k-1}+1, s_{k-1}$ for $i=\alpha_4+1,\hdots,\alpha_4+s_{\alpha_4-1}$; $w_{\alpha_4+s_{\alpha_4-1}+1}=f_4*$; $w_{\alpha_4+s_{\alpha_4-1}+2}=f_3*$; $w_{\alpha_4+s_{\alpha_4-1}+3}=f_2*$; $w_{\alpha_4+s_{\alpha_4-1}+4+j}=f_{1,\alpha_4-j-1}*$ where $j=0,\hdots, \alpha_4-2$
	\end{proof}
\begin{remark}
	Note that for $\alpha_4=2$ in theorem \ref{Hilbfunc}, we obtain  theorem 3.1 of \cite{Sahin} with $s_1=k-1$ and for $\alpha_4=3$, we obtain theorem 4.1 of \cite{Sahin2} with $s_1=s$ and $s_2=l$.
\end{remark}
\begin{theorem}\label{2ndHilb}
	The second Hilbert Series of the local ring is \\

	$Q(t)=\displaystyle\sum_{j=0}^{\alpha_3+\alpha_{21}-1}t^j\sum_{j=0}^{(\alpha_4-1)\alpha_2-\alpha_4}t^{j+\alpha_4-1}+t^{(\alpha_4-1)\alpha_2}\sum_{j=0}^{(\alpha_4-1)s_{\alpha_4-1}\alpha_2}t^j+t^{\alpha_{21}}\sum_{j=0}^{\alpha_3-1}t^j\sum_{j=0}^{\alpha_4-2}t^j-\sum_{j=0}^{(\alpha_4-1)\alpha_2-1}t^j\left[ \displaystyle\sum_{i=1}^{\alpha_4-1}\displaystyle\sum_{j=s_{i-1}}^{s_i-1}\right. $ $\displaystyle\left. t^{j\alpha_1+\left[ (\alpha_4-1)j+(i+1)\right] \alpha_{21}+\alpha_3-j} \right]+\displaystyle\sum_{j=0}^{\alpha_{21}-1}t^j\left[  \sum_{j=0}^{\alpha_3-1}t^j\sum_{j=0}^{\alpha_4-2}t^j-t^{(\alpha_4-2)\alpha_2+2}\sum_{j=0}^{\alpha_2+\alpha_3-3}t^j- \sum_{j=0}^{\alpha_4-2}t^{((\alpha_4-1)s_j+j)\alpha_2+\alpha_4-j}-\right. $ $\displaystyle\left. \sum_{j=0}^{\alpha_3-2}t^j \sum_{j=1}^{\alpha_4-2}t^{j\alpha_2+\alpha_4-j}\right] $
\end{theorem}
\begin{proof}
	Using theorem \ref{Hilbfunc}, observe that 
	
	$I_S*=(1-t)P_1(t)$ where $P_1(t)=(1-t^{\alpha_{21}+1})\left[ \displaystyle\sum_{j=0}^{\alpha_4-2}t^j-t^{\alpha_3}\sum_{j=0}^{(\alpha_4-1)\alpha_2-1}t^j\right]+t^{\alpha_4-1}\left( 1-t^{(\alpha_4-1)(\alpha_2-1)+1}\right)  $
	$-(1-t^{\alpha_{21}})\displaystyle\sum_{j=0}^{\alpha_4-2}t^{j\alpha_2+\alpha_4-j}\left( 1-t^{\alpha_3-1}\sum_{j=0}^{\alpha_2-1}t^j\right)-t^{((\alpha_4-1)s_{\alpha_4-1}+\alpha_4-1)\alpha_2+1}(1-t) -(1-t)(1-t^{\alpha_{21}})\displaystyle\sum_{j=0}^{\alpha_4-2}t^{((\alpha_4-1)s_j+j)\alpha_2+\alpha_4-j}$
	$-\left( 1-t^{(\alpha_4-1)\alpha_2}\right) (1-t)\left[ \displaystyle\sum_{i=1}^{\alpha_4-1}\displaystyle\sum_{j=s_{i-1}}^{s_i-1}t^{j\alpha_1+\left[ (\alpha_4-1)j+(i+1)\right] \alpha_{21}+\alpha_3-j}\right] $.
	
	\noindent$P_1(t)=(1-t)P_2(t)$ where $P_2(t)=(1-t^{\alpha_3+\alpha_{21}})\displaystyle\sum_{j=0}^{(\alpha_4-1)\alpha_2-\alpha_4}t^{j+\alpha_4-1}-(1-t^{\alpha_2+\alpha_3-2})t^{(\alpha_4-2)\alpha_2+2}\sum_{j=0}^{\alpha_{21}-1}t^j$
	$-(1-t^{\alpha_3-1})\displaystyle\sum_{j=0}^{\alpha_{21}-1}t^j\sum_{j=1}^{\alpha_4-2}t^{j\alpha_2+\alpha_4-j}+t^{(\alpha_4-1)\alpha_2}(1-t^{(\alpha_4-1)s_{\alpha_4-1}\alpha_2+1})+(1-t^{\alpha_3})\sum_{j=0}^{\alpha_{21}}t^j\sum_{j=0}^{\alpha_4-2}t^j-(1-t^{\alpha_{21}})\sum_{j=0}^{\alpha_4-2}t^{((\alpha_4-1)s_j+j)\alpha_2+\alpha_4-j}$
	$-\left( 1-t^{(\alpha_4-1)\alpha_2}\right) \left[ \displaystyle\sum_{i=1}^{\alpha_4-1}\displaystyle\sum_{j=s_{i-1}}^{s_i-1}t^{j\alpha_1+\left[ (\alpha_4-1)j+(i+1)\right] \alpha_{21}+\alpha_3-j} \right] $. Here, $P_2(t)=(1-t)Q(t)$ where 
	
	$Q(t)=\displaystyle\sum_{j=0}^{\alpha_3+\alpha_{21}-1}t^j\sum_{j=0}^{(\alpha_4-1)\alpha_2-\alpha_4}t^{j+\alpha_4-1}+t^{(\alpha_4-1)\alpha_2}\sum_{j=0}^{(\alpha_4-1)s_{\alpha_4-1}\alpha_2}t^j+t^{\alpha_{21}}\sum_{j=0}^{\alpha_3-1}t^j\sum_{j=0}^{\alpha_4-2}t^j-\sum_{j=0}^{(\alpha_4-1)\alpha_2-1}t^j\left[ \displaystyle\sum_{i=1}^{\alpha_4-1}\displaystyle\sum_{j=s_{i-1}}^{s_i-1}\right. $ $\displaystyle\left. t^{j\alpha_1+\left[ (\alpha_4-1)j+(i+1)\right] \alpha_{21}+\alpha_3-j} \right]+\displaystyle\sum_{j=0}^{\alpha_{21}-1}t^j\left[  \sum_{j=0}^{\alpha_3-1}t^j\sum_{j=0}^{\alpha_4-2}t^j-t^{(\alpha_4-2)\alpha_2+2}\sum_{j=0}^{\alpha_2+\alpha_3-3}t^j- \sum_{j=0}^{\alpha_4-2}t^{((\alpha_4-1)s_j+j)\alpha_2+\alpha_4-j}-\right. $ $\displaystyle\left. \sum_{j=0}^{\alpha_3-2}t^j \sum_{j=1}^{\alpha_4-2}t^{j\alpha_2+\alpha_4-j}\right] $
	\end{proof}
\noindent To prove the Hilbert function is nondecreasing, we need to show that there are no negative terms in $Q(t)$.

\begin{corollary}\label{simplified}
	$Q(t)$ could be simplified as 
	
	$Q(t)=\displaystyle t^{\alpha_4-1}\sum_{j=0}^{\alpha_3+\alpha_{21}-1}t^j\sum_{j=0}^{\alpha_2-1}t^j+\displaystyle t^{\alpha_4+\alpha_2-1+\alpha_{21}}\sum_{j=0}^{\alpha_2-\alpha_{21}-2}t^j \sum_{j=0}^{\alpha_3-2}t^j\sum_{j=0}^{\alpha_4-3}t^{j(\alpha_2-1)}+t^{(\alpha_4-1)\alpha_2}\sum_{j=0}^{(\alpha_4-1)s_{\alpha_4-1}\alpha_2}t^j+$\newline
	$t^{\alpha_{21}}\displaystyle\sum_{j=0}^{\alpha_3-1}t^j\displaystyle \sum_{j=0}^{\alpha_4-2}t^j-\sum_{j=0}^{(\alpha_4-1)\alpha_2-1}t^j\left[ \displaystyle\sum_{i=1}^{\alpha_4-1}\displaystyle\sum_{j=s_{i-1}}^{s_i-1} t^{j\alpha_1+\left[ (\alpha_4-1)j+(i+1)\right] \alpha_{21}+\alpha_3-j} \right]+\displaystyle\sum_{j=0}^{\alpha_{21}-1}t^j \sum_{j=0}^{\alpha_3-1}t^j\sum_{j=0}^{\alpha_4-2}t^j $
	\newline $ -\displaystyle\sum_{j=0}^{\alpha_{21}-1}t^j\sum_{j=0}^{\alpha_4-2}t^{((\alpha_4-1)s_j+j)\alpha_2+\alpha_4-j}+t^{\alpha_4+\alpha_2+\alpha_3-2}\displaystyle\sum_{j=0}^{\alpha_{21}-1}t^j\sum_{j=0}^{(\alpha_4-3)(\alpha_2-1)-\alpha_3}t^j+  \displaystyle t^{\alpha_4+\alpha_2+\alpha_3+\alpha_{21}-2}\sum_{j=0}^{(\alpha_4-3)\alpha_2-\alpha_3-\alpha_{21}-2}t^j$
\end{corollary}
\begin{proof}
	See the appendix.
\end{proof}

\begin{theorem}\label{nondecreasing}
Let the notation be as in theorem \ref{standardbasis}. Then the local ring has nondecreasing Hilbert function.	
\end{theorem}

\begin{proof}
Once we do the necessary cancellations, there are no negative terms in $Q(t)$.


	We will prove this in a couple of steps. Our aim is to show that the two negative sums (5th and 7th terms) in proposition \ref{simplified} will be cancelled out by the other terms in $Q(t)$. 
	\begin{enumerate}
		\item[STEP 1] Term 1,2,3 and 6 contains all of $t^j$ where $0\leq j \leq (\alpha_4-1)(s_{\alpha_4-1}+1)\alpha_2$ with a positive sign.
		\item[STEP 2] 5th term contains only SOME of  $t^{j}$ where $\alpha_{21}+\alpha_3 \leq j \leq (s_{\alpha_4-1}-1)\alpha_1+\left[ (\alpha_4-1)(s_{\alpha_4-1}-1)+\alpha_4\right]\alpha_{21}+\alpha_3-s_{\alpha_4-1}+(\alpha_4-1)\alpha_2 $ (where the coefficients are not grater than 1)
		
		\noindent To see the coefficients are not greater than 1, it is enough observe that the difference in between the degrees of the two consecutive  terms in $ \displaystyle\sum_{i=1}^{\alpha_4-1}\displaystyle\sum_{j=s_{i-1}}^{s_i-1} t^{j\alpha_1+\left[ (\alpha_4-1)j+(i+1)\right] \alpha_{21}+\alpha_3-j}$ are either $\alpha_1+(\alpha_4-1)\alpha_{21}-1$(for a fixed i) or $\alpha_1+\alpha_4\alpha_{21}-1$(when i changes) and both of these degrees are greater than $(\alpha_4-1)\alpha_2-1$. To see this, observe that  lemma \ref{lmf1j} gives $(\alpha_4-1)\alpha_2<\alpha_1+(\alpha_4-1)\alpha_{21}$, by taking $j=\alpha_4-1$.

		\item[STEP 3] 7th term contains only SOME of  $t^{j}$ where $\alpha_4 \leq j \leq \left[ (\alpha_4-1)(s_{\alpha_4-2}+1)-1\right]\alpha_{2}+\alpha_{21}+1$ (where the coefficients are not grater than 1)
		
		\noindent To see the coefficients are not greater than 1, it is enough observe that the difference in between the degrees of the two consecutive  terms in $\displaystyle\sum_{j=0}^{\alpha_4-2}t^{((\alpha_4-1)s_j+j)\alpha_2+\alpha_4-j}$ are greater than $\alpha_{21}-1$ but this is clear as $\alpha_2>\alpha_{21}+1$.

		\item[STEP 4] None of the $t^j$'s in the 5th term are in the 7th term so their sum is only the sum of $t^j$'s where $j\leq (s_{\alpha_4-1}-1)\alpha_1+\left[ (\alpha_4-1)(s_{\alpha_4-1}-1)+\alpha_4\right]\alpha_{21}+\alpha_3-s_{\alpha_4-1}+(\alpha_4-1)\alpha_2$ .\\

		\noindent To see none of the $t^j$'s in the 5th term are in the  7th term, observe first \\

	\noindent	$\displaystyle\sum_{j=0}^{(\alpha_4-1)\alpha_2-1}t^j\left[ \displaystyle\sum_{i=1}^{\alpha_4-1}\displaystyle\sum_{j=s_{i-1}}^{s_i-1} t^{j\alpha_1+\left[ (\alpha_4-1)j+(i+1)\right] \alpha_{21}+\alpha_3-j} \right]=$\newline$\displaystyle\sum_{i=1}^{\alpha_4-1}\sum_{j=(\alpha_1+(\alpha_4-1)\alpha_{21}-1)s_{i-1}}^{(\alpha_1+(\alpha_4-1)\alpha_{21}-1)s_i-1}t^{j+(i+1)\alpha_{21}+\alpha_3}-\displaystyle\sum_{j=(\alpha_4-1)\alpha_2}^{\alpha_1-(\alpha_4-1)\alpha_{21}-2}t^j \left[ \displaystyle\sum_{i=1}^{\alpha_4-1}\displaystyle\sum_{j=s_{i-1}}^{s_i-1} t^{j\alpha_1+\left[ (\alpha_4-1)j+(i+1)\right] \alpha_{21}+\alpha_3-j} \right].$

\noindent	Hence it is enough to show that none of the terms in 
		$Y(t)=\displaystyle\sum_{j=0}^{\alpha_{21}-1}t^j\sum_{i=0}^{\alpha_4-2}t^{((\alpha_4-1)s_i+i)\alpha_2+\alpha_4-i}$ are in  \newline $Z(t)=\displaystyle\sum_{i=1}^{\alpha_4-1}\sum_{j=(\alpha_1+(\alpha_4-1)\alpha_{21}-1)s_{i-1}}^{(\alpha_1+(\alpha_4-1)\alpha_{21}-1)s_i-1}t^{j+(i+1)\alpha_{21}+\alpha_3}$.\\

\noindent Observe that for a fixed $i$, by the definition of $s_i$,
		\begin{eqnarray*}
			((\alpha_4-1)s_i+i)\alpha_2+\alpha_4-i+\alpha_{21}-1 & < & s_i\alpha_1+((\alpha_4-1)s_i+i+2)\alpha_{21}+\alpha_3-s_i\\
			(\alpha_4-1)\alpha_2s_i+i(\alpha_2-1)+\alpha_4+\alpha_{21}-1&<&(\alpha_1+(\alpha_4-1)\alpha_{21}-1)s_i+(i+1)\alpha_{21}+\alpha_{21}+\alpha_3
		\end{eqnarray*}
\noindent	which shows that the degree of the last term of $Y(t)$  for $i$, is less than the degree of the first term of  $Z(t)$   for $i+1$.\\
	
\noindent	On the other hand, by the definition $s_i$,
	\begin{eqnarray*}
		s_{i-1}(\alpha_1-1+(\alpha_4-1)(-\alpha_2+\alpha_{21}))+\alpha_{21}+\alpha_3-\alpha_4&\leq& (\alpha_2-\alpha_{21}-1)i\\
		s_{i-1}(\alpha_1-1+(\alpha_4-1)(-\alpha_2+\alpha_{21}))+\alpha_{21}+\alpha_3+(\alpha_4-1)\alpha_2s_i+\alpha_{21}i&\leq& (\alpha_2-1)i+\alpha_4+(\alpha_4-1)\alpha_2s_i	
	\end{eqnarray*}
Since $s_{i-1}\leq s_i$ for any $i$, this implies,
$$(\alpha_1+(\alpha_4-1)\alpha_{21}-1)s_{i-1}-1+(i+1)\alpha_{21}+\alpha_3<(\alpha_4-1)\alpha_2s_i+i(\alpha_2-1)+\alpha_{21}$$
which shows that the degree last term of $Z(t)$ for $i-1$ is less than the degree of the first term of $Y(t)$ for $i$.

		\item[STEP 5] Maximum power appearing in the sum of 1st, 2nd, 3rd and 6th terms is greater than the maximum power apperaring in the sum of 5th and 7th terms.
		
		\noindent Indeed, by the definition of $s_{\alpha_4-1}$, $s_{\alpha_4-1}$ satisfies,
		\begin{eqnarray*}
			(\alpha_4-1)s_{\alpha_{4-1}}\alpha_2+1&\geq&(s_{\alpha_4-1}-1)(\alpha_1-1)+((\alpha_4-1)(s_{\alpha_4-1}-1)+\alpha_4)\alpha_{21}+\alpha_3
		\end{eqnarray*}
	Then adding $(\alpha_4-1)\alpha_2-1$ to both sides, we obtain the result.
		
	\end{enumerate}
Hence, when we add 1st,2nd,6th, 7th 9th and 10th terms, all of the negative ones will disappear and there won't be any negative terms and the Hilbert Function will be non-decreasing.
\end{proof}


\section{Examples}\label{4}

The following examples are verified using the computer algebra system SINGULAR, see  \cite{singular}.
\begin{example}
	Let $\alpha_{21}=5$,  $\alpha_{1}=21$,  $\alpha_{2}=11$,  $\alpha_{3}=7$,  $\alpha_{4}=4$. One can easily check that  $n_1=232<n_2=237<n_3=531<n_4=1447$  and hence the conditions $(1)$, $(2)$, $(3)$  are automatically satisfied and also  $(4)$ is satisfied. Which implies that our theorem is applicable. Using the definition of $s_j$, $s_0=0$, $s_1=0$, $s_2=2$ and $s_3=4$. According to the theorem \ref{standardbasis}, a standard basis for the defining ideal is $\{f_{1,0},f_{1,1},f_{1,2},f_{1,3},f_2,f_3,f_4,g_{0,0}, g_{1,0},g_{2,0},g_{2,1},g_{2,2}, g_{3,2},g_{3,3},g_{3,4}\}$. Indeed, the standard basis that SINGULAR gives is 
	
\noindent	$I_S=\{ g_{0,0}=X_2X_4^3-X_1^6X_3^6,
	f_{1,0}=X_3X_4^3-X_1^{21},
	f_4=X_4^4-X_1X_2^{10}X_3^6,
	f_2=X_1^5X_4-X_2^{11},
	f_3=X_3^7-X1^{15}X_2,
	g_{1,0}=X_2^{12}X_4^2-X_1^{11}X_3^6,
	f_{1,1}=X_2^{11}X_3X_4^2-X_1^{26},
	g_{2,0}=X_1^{16}X_3^6-X_2^{23}X_4,
	f_{1,2}=X_2^{22}X_3X_4-X_1^{31},
	f_{1,3}=X_2^{33}X_3-X_1^{36},
	g_{2,1}=X_1^{52}X_3^5-X_2^{56}X_4,
	g_{2,2}=X_2^{89}X_4-X_1^{88}X_3^4,
	g_{3,2}=X_1^{93}X_3^4-X_2^{100},
	g_{3,3}=X_1^{129}X_3^3-X_2^{133},
	g_{3,4}=X_2^{166}-X_1^{165}X_3^2\}$
the  numerator of the Hilbert series of the local ring is
$P(I_S*)=  1 -3 t^4+3 t^5-2 t^6- t^7+ 3 t^9 -2 t^{10} + t^{11}+ t^{13}-2 t^{14}+2 t^{15}- t^{16}+2 t^{19}-2 t^{20}- t^{22}+2 t^{23}-2 t^{24}+ t^{26}+t^{29}- t^{31}- t^{34}+ t^{36}+ t^{40}- t^{41}+t^{55}-2 t^{56}+2 t^{58}-t^{59}+ t^{95}-2 t^{96}+2 t^{98}- t^{99}+ t^{130}-2 t^{131}+2 t^{133}- t^{134}+t^{165}-3 t^{166}+3 t^{167} -t^{168}$ and the second Hilbert series is 
$Q(t)= 1 +3 t^1+  6 t^2+ 10 t^3+ 12 t^4+15 t^5+   17 t^6 +17 t^7+15 t^8+14 t^9+12 t^{10}+10 t^{11}+8 t^{12}+7 t^{13}+5 t^{14}+4 t^{15}+3 t^{16}+2 t^{17}+ t^{18}+2 t^{19}+3 t^{20}+4 t^{21}+4 t^{22}+5 t^{23}+5 t^{24}+4 t^{25}+3 t^{26}+2 t^{27}+ t^{28}+ t^{29}+2 t^{30}+3 t^{31}+4 t^{32}+5 t^{33}+5 t^{34}+4 t^{35}+3 t^{36}+2 t^{37}+ t^{38}+ t^{55}+ t^{56}+ t^{95}+ t^{96}+ t^{130}+ t^{131} + t^{165}$.
\end{example}
\begin{example}
	Let $\alpha_{21}=10$,  $\alpha_{1}=60$,  $\alpha_{2}=20$,  $\alpha_{3}=8$,  $\alpha_{4}=6$. In this case, $n_1=801<n_2=831<n_3=5010<n_4=8610$    
	and our theorem is applicable. Using the definition of $s_j$, $s_0=0$, $s_1=0$, $s_2=1$, $s_3=1$, $s_4=3$ and $s_5=4$. According to the theorem \ref{standardbasis}, a standard basis for the defining ideal $I_S$ is $\left( f_{1,0},f_{1,1},f_{1,2},f_{1,3}, f_{1,4}, f_{1,5}, f_2, f_3, f_4, g_{0,0},  g_{1,0}, g_{2,0},  g_{2,1}, g_{3,1},\right. $ $ \left.  g_{3,2}, g_{4,2}, g_{4,3}, g_{5,3}, g_{5,4}\right) $. Indeed, the standard basis that SINGULAR gives is 
	
	\noindent	$I_S=\{g_{0,0}=X_2X_4^5-X_1^{11}X_3^7
,f_{1,0}=X_3X_4^5-X_1^{60}
,f_4=X_4^6-X_1X_2^{19}X_3^7
,f_3=X_3^8-X_1^{49}X_2
,f_2=X_1^{10}X_4-X_2^{20}
,g_{1,0}=X_2^{21}X_4^4-X_1^{21}X_3^7
,f_{1,1}=X_2^{20}X_3X_4^4-X_1^{70}
,g_{2,0}=X_1^{31}X_3^7-X_2^{41}X_4^3
,f_{1,2}=X_2^{40}X_3X_4^3-X_1^{80}
,f_{1,3}=X_2^{60}X_3X_4^2-X_1^{90}
,f_{1,4}=X_2^{80}X_3X_4-X_1^{100}
,f_{1,5}=X_2^{100}X_3-X_1^{110}
,g_{2,1}=X_2^{141}X_4^3-X_1^{141}X_3^6
,g_{3,1}=X_1^{151}X_3^6-X_2^{161}X_4^2
,g_{3,2}=X_2^{261}X_4^2-X_1^{261}X_3^5
,g_{4,2}=X_1^{271}X_3^5-X_2^{281}X_4
,g_{4,3}=X_2^{381}X_4-X_1^{381}X_3^4
,g_{5,3}=X_1^{391}X_3^4-X_2^{401}
,g_{5,4}=X_2^{501}-X_1^{501}X_3^3\}$. The  numerator of the Hilbert series of the local ring is
$P(I_S*)=     1 -3 t^{6}+3 t^7-2 t^8-t^{11}+t^{13}+3 t^{16}-3 t^{17}+t^{18}+t^{19}-t^{23}-2 t^{25}+3 t^{26}- t^{27}+t^{32}-t^{33}+2 t^{35}-3 t^{36}+t^{37}-t^{38}+2 t^{39} -t^{40}-t^{42}+t^{43}-t^{44}+t^{45}+t^{51}-t^{52}+t^{54}-t^{55}-t^{61}+t^{62}-t^{63}+t^{64}+t^{70}-t^{71}+t^{73}-t^{74}-t^{80}+t^{81}-t^{82}+t^{83}+t^{89}-t^{90}+t^{92}-t^{93}-t^{99}+t^{100}-t^{101}+t^{102}+t^{108}-t^{109}+t^{138}-2 t^{139}+t^{140}-t^{144}+2 t^{145}- t^{146}+t^{154}-2 t^{155}+t^{156}-t^{157}+2 t^{158}-t^{159}+t^{257}-2 t^{258}+t^{259}-t^{263}+2 t^{264}-t^{265}+t^{273}-2 t^{274}+t^{275}-t^{276}+2 t^{277}-t^{278}+t^{376}-2 t^{377}+t^{378}-t^{382}+2 t^{383}-t^{384}+t^{392}-2 t^{393}+t^{394}-t^{395}+2 t^{396}-t^{397}+t^{495}-2 t^{496}+t^{497}-t^{501}+2 t^{502}-t^{503}$
and the second Hilbert series is 
$Q(t)= 1 + 3 t  + 6 t^2+10 t^3   +15 t^4     +21 t^5+25 t^6 +30 t^7      +34 t^8       +37 t^9       +39 t^{10}    +    39 t^{11}        +37 t^{12}     +  34 t^{13}     +  30 t^{14}       +25 t^{15}       + 22 t^{16}       +18 t^{17}        +14 t^{18}        +11 t^{19}        +9 t^{20}   +8 t^{21} +8 t^{22}+8 t^{23}+8 t^{24}+6 t^{25}+5 t^{26}+4 t^{27}+3 t^{28}+2 t^{29}+t^{30}+2 t^{35}+3 t^{36}+4 t^{37}+4 t^{38}+5 t^{39}+6 t^{40}+7 t^{41}+7 t^{42}+7 t^{43}+6 t^{44}+5 t^{45}+4 t^{46}+3 t^{47}+2 t^{48}+t^{49}+t^{54}+2 t^{55}+3 t^{56}+4 t^{57}+5 t^{58}+6 t^{59}+7 t^{60}+7 t^{61}+7 t^{62}+6 t^{63}+5 t^{64}+4 t^{65}+3 t^{66}+2 t^{67}+t^{68}+ t^{73}+2 t^{74}+3 t^{75}+4 t^{76}+5 t^{77}+6 t^{78}+7 t^{79}+7 t^{80}+7 t^{81}+6 t^{82}+5 t^{83}+4 t^{84}+3 t^{85}+2 t^{86}+t^{87}+ t^{92}+2 t^{93}+3 t^{94}+4 t^{95}+5 t^{96}+6 t^{97}+7 t^{98}+7 t^{99}+7 t^{100}+6 t^{101}+5 t^{102}+4 t^{103}+3 t^{104}+2 t^{105}+t^{106}
 +t^{138}+t^{139}+t^{140}+t^{141}+t^{142}+t^{143}+t^{154}+t^{155}+t^{156}+t^{257}+t^{258}+t^{259}+t^{260}+t^{261}+ t^{262}+t^{273}+t^{274}+t^{275}+t^{376}+t^{377}+t^{378}+t^{379}+t^{380}+t^{381}+t^{392}+t^{393}+t^{394}+t^{495}+t^{496}+t^{497}+t^{498}+t^{499}+t^{500}$

\end{example}
\section{Conclusion}
If $n_1<n_2<n_3<n_4$ and $<n_1,n_2,n_3,n_4>$ is a 4 generated pseudo-symmetric numerical semigroup, then the  Hilbert function of the local ring is always nondecreasing. This supports Rossi's conjecture, saying that " The Hilbert function of a one dimensional Cohen-Macaulay local ring with small enough embedding dimension is nondecrasing."
\section{Apppendix}
In this appendix we show the technical details to prove the Corollary \ref{simplified}. 
\begin{remark}\label{simp1}
	Let $R_1(t)=\displaystyle\sum_{j=0}^{\alpha_3+\alpha_{21}-1}t^j\sum_{j=0}^{(\alpha_4-1)\alpha_2-\alpha_4}t^{j+\alpha_4-1}$ and $R_2(t)=\displaystyle t^{(\alpha_4-2)\alpha_2+2}\sum_{j=0}^{\alpha_{21}-1}t^j\sum_{j=0}^{\alpha_2+\alpha_3-3}t^j$. Then $R_1(t)-R_2(t)=t^{\alpha_4-1}\displaystyle\sum_{j=0}^{\alpha_3+\alpha_{21}-1}t^j\sum_{j=0}^{(\alpha_4-2)(\alpha_2-1)}t^j+t^{(\alpha_4-2)\alpha_2+\alpha_{21}+2}\sum_{j=0}^{\alpha_3-1}t^j\sum_{j=0}^{\alpha_2-\alpha_{21}-3}t^j$
\end{remark}
\begin{proof}
	
	Observe that \begin{eqnarray*}
		R_1(t)&=&\left[ \displaystyle\sum_{j=0}^{\alpha_3+\alpha_{21}-1}t^j\sum_{j=0}^{(\alpha_4-2)(\alpha_2-1)}t^{j} + \displaystyle\sum_{j=0}^{\alpha_3+\alpha_{21}-1}t^j\sum_{j=(\alpha_4-2)(\alpha_2-1)+1}^{(\alpha_4-1)\alpha_2-\alpha_4}t^{j} \right] t^{\alpha_4-1}\\
		&=&\left[\displaystyle\sum_{j=0}^{\alpha_3+\alpha_{21}-1}t^j\sum_{j=0}^{(\alpha_4-2)(\alpha_2-1)}t^{j} \right] t^{\alpha_4-1}+t^{\alpha_4-1+(\alpha_4-2)(\alpha_2-1)+1} \left[  \displaystyle\sum_{j=0}^{\alpha_3+\alpha_{21}-1}t^j\sum_{j=0}^{\alpha_2-3}t^{j}\right] \\
		&=& \left[\displaystyle\sum_{j=0}^{\alpha_3+\alpha_{21}-1}t^j\sum_{j=0}^{(\alpha_4-2)(\alpha_2-1)}t^{j} \right] t^{\alpha_4-1}+t^{(\alpha_4-2)\alpha_2+2} \left[  \displaystyle\sum_{j=0}^{\alpha_3+\alpha_{21}-1}t^j\sum_{j=0}^{\alpha_2-3}t^{j}\right].
	\end{eqnarray*}	Let $S_1(t)=  \left[\displaystyle\sum_{j=0}^{\alpha_3+\alpha_{21}-1}t^j\sum_{j=0}^{(\alpha_4-2)(\alpha_2-1)}t^{j} \right] t^{\alpha_4-1}$ Then 
	
	\begin{eqnarray*}
		R_1(t)-R_2(t)&=&S_1(t)+t^{(\alpha_4-2)\alpha_2+2} \left[  \displaystyle\sum_{j=0}^{\alpha_3+\alpha_{21}-1}t^j\sum_{j=0}^{\alpha_2-3}t^{j}-\sum_{j=0}^{\alpha_{21}-1}t^j\sum_{j=0}^{\alpha_2+\alpha_3-3}t^j\right]\\
		&=&S_1(t)+t^{(\alpha_4-2)\alpha_2+2} \left[  \displaystyle\sum_{j=0}^{\alpha_{21}-1}t^j\sum_{j=0}^{\alpha_2-3}t^{j}+\sum_{j=\alpha_{21}}^{\alpha_3+\alpha_{21}-1}t^j\sum_{j=0}^{\alpha_2-3}t^j-\sum_{j=0}^{\alpha_{21}-1}t^j\sum_{j=0}^{\alpha_2-3}t^j-\sum_{j=0}^{\alpha_{21}-1}t^j\sum_{j=\alpha_2-2}^{\alpha_2+\alpha_3-3}t^j\right]\\
		&=&S_1(t)+t^{(\alpha_4-2)\alpha_2+2} \left[  \displaystyle\sum_{j=\alpha_{21}}^{\alpha_3+\alpha_{21}-1}t^j\sum_{j=0}^{\alpha_2-3}t^j-\sum_{j=0}^{\alpha_{21}-1}t^j\sum_{j=\alpha_2-2}^{\alpha_2+\alpha_3-3}t^j\right]\\
		&=&S_1(t)+t^{(\alpha_4-2)\alpha_2+2} \left[ t^{\alpha_{21}} \displaystyle\sum_{j=0}^{\alpha_3-1}t^j\sum_{j=0}^{\alpha_2-3}t^j-t^{\alpha_2-2}\sum_{j=0}^{\alpha_{21}-1}t^j\sum_{j=0}^{\alpha_3-1}t^j\right]\\
		&=&S_1(t)+t^{(\alpha_4-2)\alpha_2+\alpha_{21}+2}\displaystyle\sum_{j=0}^{\alpha_3-1} t^j\left[ \sum_{j=0}^{\alpha_2-3}t^j-t^{\alpha_2-\alpha_{21}-2}\sum_{j=0}^{\alpha_{21}-1}t^j\right]\\
		&=&S_1(t)+t^{(\alpha_4-2)\alpha_2+\alpha_{21}+2}\displaystyle\sum_{j=0}^{\alpha_3-1} t^j\sum_{j=0}^{\alpha_2-\alpha_{21}-3}t^j
	\end{eqnarray*}
	Hence,
	\begin{equation}\label{R(t)}
		R_1(t)-R_2(t)=\left[\displaystyle\sum_{j=0}^{\alpha_3+\alpha_{21}-1}t^j\sum_{j=0}^{(\alpha_4-2)(\alpha_2-1)}t^{j} \right] t^{\alpha_4-1}+t^{(\alpha_4-2)\alpha_2+\alpha_{21}+2}\displaystyle\sum_{j=0}^{\alpha_3-1} t^j\sum_{j=0}^{\alpha_2-\alpha_{21}-3}t^j
	\end{equation}

\end{proof}
\begin{corollary}
	Using equation \ref{R(t)}, we can rewrite $Q(t)$ as,

	$Q(t)= t^{\alpha_4-1}\displaystyle\sum_{j=0}^{\alpha_3+\alpha_{21}-1}t^j\sum_{j=0}^{(\alpha_4-2)(\alpha_2-1)}t^j+t^{(\alpha_4-2)\alpha_2+\alpha_{21}+2}\sum_{j=0}^{\alpha_3-1}t^j\sum_{j=0}^{\alpha_2-\alpha_{21}-3}t^j+t^{(\alpha_4-1)\alpha_2}\sum_{j=0}^{(\alpha_4-1)s_{\alpha_4-1}\alpha_2}t^j+$ \newline
	$t^{\alpha_{21}}\displaystyle\sum_{j=0}^{\alpha_3-1}t^j\displaystyle \sum_{j=0}^{\alpha_4-2}t^j-\sum_{j=0}^{(\alpha_4-1)\alpha_2-1}t^j\left[ \displaystyle\sum_{i=1}^{\alpha_4-1}\displaystyle\sum_{j=s_{i-1}}^{s_i-1} t^{j\alpha_1+\left[ (\alpha_4-1)j+(i+1)\right] \alpha_{21}+\alpha_3-j} \right]$\newline$+\displaystyle\sum_{j=0}^{\alpha_{21}-1}t^j\left[  \sum_{j=0}^{\alpha_3-1}t^j\sum_{j=0}^{\alpha_4-2}t^j-\right. $ $\displaystyle\left. \sum_{j=0}^{\alpha_3-2}t^j \sum_{j=1}^{\alpha_4-2}t^{j\alpha_2+\alpha_4-j}\right] -\displaystyle\sum_{j=0}^{\alpha_{21}-1}t^j\sum_{j=0}^{\alpha_4-2}t^{((\alpha_4-1)s_j+j)\alpha_2+\alpha_4-j}$
\end{corollary}
\begin{remark}
	Let $S_1(t)=t^{\alpha_4-1}\displaystyle\sum_{j=0}^{\alpha_3+\alpha_{21}-1}t^j\sum_{j=0}^{(\alpha_4-2)(\alpha_2-1)}t^j$  as in the the proof of remark \ref{simp1} and let
	
	$S_2(t)=\displaystyle\sum_{j=0}^{\alpha_{21}-1}t^j   \sum_{j=0}^{\alpha_3-2}t^j \sum_{j=1}^{\alpha_4-2}t^{j\alpha_2+\alpha_4-j}$. Then
	
	\begin{equation*}
		\begin{aligned}
			S_1(t)-S_2(t)=&\displaystyle t^{\alpha_4-1}\sum_{j=0}^{\alpha_3+\alpha_{21}-1}t^j\sum_{j=0}^{\alpha_2-1}t^j+t^{\alpha_4+\alpha_2+\alpha_3-2}\sum_{j=0}^{\alpha_{21}}t^j\sum_{j=0}^{(\alpha_4-3)(\alpha_2-1)-1}t^j\\& +\displaystyle t^{\alpha_4+\alpha_2-1+\alpha_{21}}\sum_{j=0}^{\alpha_2-\alpha_{21}-2}t^j \sum_{j=0}^{\alpha_3-2}t^j\sum_{j=0}^{\alpha_4-3}t^{j(\alpha_2-1)}-t^{(\alpha_4-2)\alpha_2+2} \displaystyle\sum_{j=0}^{\alpha_3-2}t^j \sum_{j=0}^{\alpha_2-2}t^j	
		\end{aligned}
	\end{equation*}

\end{remark}
\begin{proof}
	Observe that 
	\begin{eqnarray*}
		S_1(t)&=&\displaystyle t^{\alpha_4-1}\sum_{j=0}^{\alpha_3+\alpha_{21}-1}t^j\sum_{j=0}^{\alpha_2-1}t^j+t^{\alpha_4+\alpha_2-1}\sum_{j=0}^{\alpha_3+\alpha_{21}-1}t^j\sum_{j=0}^{(\alpha_4-3)(\alpha_2-1)-1}t^j \\
		&=&\displaystyle t^{\alpha_4-1}\sum_{j=0}^{\alpha_3+\alpha_{21}-1}t^j\sum_{j=0}^{\alpha_2-1}t^j+t^{\alpha_4+\alpha_2-1}\sum_{j=0}^{\alpha_3-2}t^j\sum_{j=0}^{(\alpha_4-3)(\alpha_2-1)-1}t^j+t^{\alpha_4+\alpha_2+\alpha_3-2}\sum_{j=0}^{\alpha_{21}}t^j\sum_{j=0}^{(\alpha_4-3)(\alpha_2-1)-1}t^j\\
		&=&S_3(t)+t^{\alpha_4+\alpha_2-1}\sum_{j=0}^{\alpha_3-2}t^j\sum_{j=0}^{(\alpha_4-3)(\alpha_2-1)-1}t^j
	\end{eqnarray*}
	where $S_3(t)=\displaystyle t^{\alpha_4-1}\sum_{j=0}^{\alpha_3+\alpha_{21}-1}t^j\sum_{j=0}^{\alpha_2-1}t^j+t^{\alpha_4+\alpha_2+\alpha_3-2}\sum_{j=0}^{\alpha_{21}}t^j\sum_{j=0}^{(\alpha_4-3)(\alpha_2-1)-1}t^j$. On the other hand,
	
	\begin{eqnarray*}
		S_2(t)&=&\displaystyle\sum_{j=0}^{\alpha_{21}-1}t^j   \sum_{j=0}^{\alpha_3-2}t^j \sum_{j=1}^{\alpha_4-2}t^{j\alpha_2+\alpha_4-j}=t^{\alpha_4+\alpha_2-1}\displaystyle\sum_{j=0}^{\alpha_{21}-1}t^j\sum_{j=0}^{\alpha_3-2}t^j\sum_{j=0}^{\alpha_4-3}t^{j(\alpha_2-1)}\\
		&=&t^{\alpha_4+\alpha_2-1}\left[ \displaystyle\sum_{j=0}^{\alpha_{2}-2}t^j-\sum_{j=\alpha_{21}}^{\alpha_2-2}t^j \right]\sum_{j=0}^{\alpha_3-2}t^j\sum_{j=0}^{\alpha_4-3}t^{j(\alpha_2-1)}\\
		&=&t^{\alpha_4+\alpha_2-1} \displaystyle\sum_{j=0}^{\alpha_{2}-2}t^j\sum_{j=0}^{\alpha_3-2}t^j\sum_{j=0}^{\alpha_4-3}t^{j(\alpha_2-1)}-t^{\alpha_4+\alpha_2-1+\alpha_{21}}\sum_{j=0}^{\alpha_2-\alpha_{21}-2}t^j \sum_{j=0}^{\alpha_3-2}t^j\sum_{j=0}^{\alpha_4-3}t^{j(\alpha_2-1)}\\
		&=&t^{\alpha_4+\alpha_2-1} \displaystyle\sum_{j=0}^{\alpha_3-2}t^j\sum_{j=0}^{(\alpha_4-2)(\alpha_2-1)-1}t^{j}-t^{\alpha_4+\alpha_2-1+\alpha_{21}}\sum_{j=0}^{\alpha_2-\alpha_{21}-2}t^j \sum_{j=0}^{\alpha_3-2}t^j\sum_{j=0}^{\alpha_4-3}t^{j(\alpha_2-1)}
	\end{eqnarray*}
	Let $S_4(t)=S_3(t)+\displaystyle t^{\alpha_4+\alpha_2-1+\alpha_{21}}\sum_{j=0}^{\alpha_2-\alpha_{21}-2}t^j \sum_{j=0}^{\alpha_3-2}t^j\sum_{j=0}^{\alpha_4-3}t^{j(\alpha_2-1)}$. Then 
	\begin{eqnarray*}
		S_1(t)-S_2(t)&=&S_4(t)+t^{\alpha_4+\alpha_2-1} \displaystyle\sum_{j=0}^{\alpha_3-2}t^j\left[ \sum_{j=0}^{(\alpha_4-3)(\alpha_2-1)-1}t^j-\sum_{j=0}^{(\alpha_4-2)(\alpha_2-1)-1}t^{j} \right]\\
		&=& S_4(t)-t^{\alpha_4+\alpha_2-1} \displaystyle\sum_{j=0}^{\alpha_3-2}t^j \sum_{j=(\alpha_4-3)(\alpha_2-1)}^{(\alpha_4-2)(\alpha_2-1)-1}t^j\\
		&=& S_4(t)-t^{(\alpha_4-2)\alpha_2+2} \displaystyle\sum_{j=0}^{\alpha_3-2}t^j \sum_{j=0}^{\alpha_2-2}t^j
	\end{eqnarray*}
	
	Then
	
	\begin{equation*}
		\begin{aligned}
			S_1(t)-S_2(t)=&\displaystyle t^{\alpha_4-1}\sum_{j=0}^{\alpha_3+\alpha_{21}-1}t^j\sum_{j=0}^{\alpha_2-1}t^j+t^{\alpha_4+\alpha_2+\alpha_3-2}\sum_{j=0}^{\alpha_{21}}t^j\sum_{j=0}^{(\alpha_4-3)(\alpha_2-1)-1}t^j\\& +\displaystyle t^{\alpha_4+\alpha_2-1+\alpha_{21}}\sum_{j=0}^{\alpha_2-\alpha_{21}-2}t^j \sum_{j=0}^{\alpha_3-2}t^j\sum_{j=0}^{\alpha_4-3}t^{j(\alpha_2-1)}-t^{(\alpha_4-2)\alpha_2+2} \displaystyle\sum_{j=0}^{\alpha_3-2}t^j \sum_{j=0}^{\alpha_2-2}t^j	
		\end{aligned}
	\end{equation*}
	
\end{proof}
\begin{corollary}
	Using the previous remark, 
	$Q(t)=\displaystyle t^{\alpha_4-1}\sum_{j=0}^{\alpha_3+\alpha_{21}-1}t^j\sum_{j=0}^{\alpha_2-1}t^j+t^{\alpha_4+\alpha_2+\alpha_3-2}\sum_{j=0}^{\alpha_{21}}t^j\sum_{j=0}^{(\alpha_4-3)(\alpha_2-1)-1}t^j +\displaystyle t^{\alpha_4+\alpha_2-1+\alpha_{21}}\sum_{j=0}^{\alpha_2-\alpha_{21}-2}t^j \sum_{j=0}^{\alpha_3-2}t^j\sum_{j=0}^{\alpha_4-3}t^{j(\alpha_2-1)}-t^{(\alpha_4-2)\alpha_2+2} \displaystyle\sum_{j=0}^{\alpha_3-2}t^j \sum_{j=0}^{\alpha_2-2}t^j	
	+t^{(\alpha_4-2)\alpha_2+\alpha_{21}+2}\sum_{j=0}^{\alpha_3-1}t^j\sum_{j=0}^{\alpha_2-\alpha_{21}-3}t^j+t^{(\alpha_4-1)\alpha_2}\sum_{j=0}^{(\alpha_4-1)s_{\alpha_4-1}\alpha_2}t^j+$
	$t^{\alpha_{21}}\displaystyle\sum_{j=0}^{\alpha_3-1}t^j\displaystyle \sum_{j=0}^{\alpha_4-2}t^j-\sum_{j=0}^{(\alpha_4-1)\alpha_2-1}t^j\left[ \displaystyle\sum_{i=1}^{\alpha_4-1}\displaystyle\sum_{j=s_{i-1}}^{s_i-1} t^{j\alpha_1+\left[ (\alpha_4-1)j+(i+1)\right] \alpha_{21}+\alpha_3-j} \right]$\newline$+\displaystyle\sum_{j=0}^{\alpha_{21}-1}t^j \sum_{j=0}^{\alpha_3-1}t^j\sum_{j=0}^{\alpha_4-2}t^j -\displaystyle\sum_{j=0}^{\alpha_{21}-1}t^j\sum_{j=0}^{\alpha_4-2}t^{((\alpha_4-1)s_j+j)\alpha_2+\alpha_4-j}$
\end{corollary}
Now we will focus on 2nd, 4th and 5th terms of $Q(t)$.
\begin{remark}
	$\displaystyle t^{\alpha_4+\alpha_2+\alpha_3-2}\sum_{j=0}^{\alpha_{21}}t^j\sum_{j=0}^{(\alpha_4-3)(\alpha_2-1)-1}t^j -t^{(\alpha_4-2)\alpha_2+2} \displaystyle\sum_{j=0}^{\alpha_3-2}t^j \sum_{j=0}^{\alpha_2-2}t^j  +t^{(\alpha_4-2)\alpha_2+\alpha_{21}+2}\sum_{j=0}^{\alpha_3-1}t^j\sum_{j=0}^{\alpha_2-\alpha_{21}-3}t^j=t^{\alpha_4+\alpha_2+\alpha_3-2}\displaystyle\sum_{j=0}^{\alpha_{21}-1}t^j\sum_{j=0}^{(\alpha_4-3)(\alpha_2-1)-\alpha_3}t^j+t^{\alpha_4+\alpha_2+\alpha_3+\alpha_{21}-2}\sum_{j=0}^{(\alpha_4-3)\alpha_2-\alpha_3-\alpha_{21}-2}t^j	$	
\end{remark}
\begin{proof}
	$\displaystyle t^{\alpha_4+\alpha_2+\alpha_3-2}\sum_{j=0}^{\alpha_{21}}t^j\sum_{j=0}^{(\alpha_4-3)(\alpha_2-1)-1}t^j -t^{(\alpha_4-2)\alpha_2+2} \displaystyle\sum_{j=0}^{\alpha_3-2}t^j \sum_{j=0}^{\alpha_2-2}t^j  +t^{(\alpha_4-2)\alpha_2+\alpha_{21}+2}\sum_{j=0}^{\alpha_3-1}t^j\sum_{j=0}^{\alpha_2-\alpha_{21}-3}t^j= \displaystyle t^{\alpha_4+\alpha_2+\alpha_3-2}\sum_{j=0}^{\alpha_{21}}t^j\sum_{j=0}^{(\alpha_4-3)(\alpha_2-1)-1}t^j+ t^{(\alpha_4-2)\alpha_2+2}\sum_{j=0}^{\alpha_3-2} t^j\left[  t^{\alpha_{21}}\sum_{j=0}^{\alpha_2-\alpha_{21}-3}t^j-\sum_{j=0}^{\alpha_2-2}t^j  \right] +t^{(\alpha_4-2)\alpha_2+\alpha_{21}+1+\alpha_3}\sum_{j=0}^{\alpha_2-\alpha_{21}-3}t^j= \displaystyle t^{\alpha_4+\alpha_2+\alpha_3-2}\sum_{j=0}^{\alpha_{21}}t^j\sum_{j=0}^{(\alpha_4-3)(\alpha_2-1)-1}t^j	+t^{(\alpha_4-2)\alpha_2+2}\sum_{j=0}^{\alpha_3-2} t^j\left[  -\sum_{j=0}^{\alpha_{21}-1}t^j-t^{\alpha_2-2} \right]  +t^{(\alpha_4-2)\alpha_2+\alpha_{21}+1+\alpha_3}\sum_{j=0}^{\alpha_2-\alpha_{21}-3}t^j=
	\displaystyle t^{\alpha_4+\alpha_2+\alpha_3-2}\sum_{j=0}^{\alpha_{21}}t^j\sum_{j=0}^{(\alpha_4-3)(\alpha_2-1)-1}t^j	-t^{(\alpha_4-2)\alpha_2+2}\sum_{j=0}^{\alpha_3-2} t^j \sum_{j=0}^{\alpha_{21}-1}t^j
	-t^{(\alpha_4-1)\alpha_2}\sum_{j=0}^{\alpha_3-2} t^j +t^{(\alpha_4-2)\alpha_2+\alpha_{21}+1+\alpha_3}\sum_{j=0}^{\alpha_2-\alpha_{21}-3}t^j= 
	\displaystyle\sum_{j=0}^{\alpha_{21}-1}t^j\left[  t^{\alpha_4+\alpha_2+\alpha_3-2}\sum_{j=0}^{(\alpha_4-3)(\alpha_2-1)-1}t^j-t^{(\alpha_4-2)\alpha_2+2}\sum_{j=0}^{\alpha_3-2}t^j\right] +t^{\alpha_4+\alpha_2+\alpha_3+\alpha_{21}-2} \sum_{j=0}^{(\alpha_4-3)(\alpha_2-1)-1}t^j-t^{(\alpha_4-1)\alpha_2}\sum_{j=0}^{\alpha_3-2} t^j+t^{(\alpha_4-2)\alpha_2+\alpha_{21}+1+\alpha_3}\sum_{j=0}^{\alpha_2-\alpha_{21}-3}t^j$ 
	
	\noindent
	$=\displaystyle\sum_{j=0}^{\alpha_{21}-1}t^j\left[  t^{\alpha_4+\alpha_2+\alpha_3-2}\left(\sum_{j=0}^{(\alpha_4-3)(\alpha_2-1)-\alpha_3}t^j+\sum_{j=(\alpha_4-3)(\alpha_2-1)-\alpha_3+1}^{(\alpha_4-3)(\alpha_2-1)-1}t^j\right)-t^{(\alpha_4-2)\alpha_2+2}\sum_{j=0}^{\alpha_3-2}t^j\right] +$\newline
	$\displaystyle t^{\alpha_4+\alpha_2+\alpha_3+\alpha_{21}-2} \sum_{j=0}^{(\alpha_4-3)(\alpha_2-1)-1}t^j-t^{(\alpha_4-1)\alpha_2}\sum_{j=0}^{\alpha_3-2} t^j+t^{(\alpha_4-2)\alpha_2+\alpha_{21}+1+\alpha_3}\sum_{j=0}^{\alpha_2-\alpha_{21}-3}t^j=$\newline$ t^{\alpha_4+\alpha_2+\alpha_3-2}\displaystyle\sum_{j=0}^{\alpha_{21}-1}t^j\sum_{j=0}^{(\alpha_4-3)(\alpha_2-1)-\alpha_3}t^j+\displaystyle t^{\alpha_4+\alpha_2+\alpha_3+\alpha_{21}-2} \sum_{j=0}^{(\alpha_4-3)(\alpha_2-1)-1}t^j-t^{(\alpha_4-1)\alpha_2}\sum_{j=0}^{\alpha_3-2} t^j+$\newline$\displaystyle t^{(\alpha_4-2)\alpha_2+\alpha_{21}+1+\alpha_3}\sum_{j=0}^{\alpha_2-\alpha_{21}-3}t^j =t^{\alpha_4+\alpha_2+\alpha_3-2}\displaystyle\sum_{j=0}^{\alpha_{21}-1}t^j\sum_{j=0}^{(\alpha_4-3)(\alpha_2-1)-\alpha_3}t^j+t^{\alpha_4+\alpha_2+\alpha_3+\alpha_{21}-2}\sum_{j=0}^{(\alpha_4-3)\alpha_2-\alpha_3-\alpha_{21}-2}t^j
	$		
\end{proof}

\end{document}